\documentclass{article}
\usepackage[utf8]{inputenc}
\usepackage{amsmath}
\usepackage{amssymb}
\usepackage{hyperref}
\usepackage[nameinlink]{cleveref}
\usepackage{graphicx}
\usepackage{float}
\usepackage{amsthm}
\newtheorem{theorem}{Theorem}
\newtheorem{lemma}{Lemma}[theorem]

\usepackage[backend=biber, style=alphabetic]{biblatex}

\newtheorem*{definition}{Definition}
\usepackage[toc,page]{appendix}
\usepackage{times}

\usepackage{tikz}
\usetikzlibrary{shapes.geometric, arrows, positioning, calc}

\tikzstyle{startstop} = [rectangle, rounded corners, minimum width=3cm, minimum height=1cm,text centered, draw=black, fill=red!30]
\tikzstyle{process} = [rectangle, minimum width=3cm, minimum height=1cm, text centered, draw=black, fill=orange!30]
\tikzstyle{decision} = [diamond, minimum width=3cm, minimum height=1cm, text centered, draw=black, fill=green!30]
\tikzstyle{arrow} = [thick,->,>=stealth]

\usepackage{subcaption}
\usepackage{pgfplots}
\pgfplotsset{compat=1.16}

\hypersetup{colorlinks=true}

\title{\textbf{Existence of Deadlock-Free Routing for Arbitrary Networks}}

\author{
  Uri Mendlovic \thanks{Google Research, \texttt{urimend@google.com}}
  \and
  Yossi Matias \thanks{Google Research, \texttt{yossi@google.com}}
}
\date{March 2025}

\begin{document}

\maketitle

\begin{abstract}
Given a network of routing nodes, represented as a directed graph, we prove the following necessary and sufficient condition for the existence of deadlock-free message routing: The directed graph must contain two edge-disjoint directed trees rooted at the same node, one tree directed into the root node and the other directed away from the root node.

While the sufficiency of this condition is known, its necessity, to the best of our knowledge, has not been previously recognized or proven. Although not directly applicable to the construction of deadlock-free routing schemes, this result provides a fundamental insight into the nature of deadlock-free networks and may lead to the development of improved tools for designing and verifying such schemes.
\end{abstract}

\section{Background}
In their seminal work \cite{seminal}, Dally and Seitz introduced a model for deadlock avoidance. In their model, the network is defined as a strongly connected directed graph $G=(V, E)$ where the vertices $V$ are the compute nodes and the directed edges $E$ are directed communication channels between nodes. Messages are routed through the network using a deterministic routing scheme, depending only on the source and destination of the message.

Given a routing scheme, we define a dependency from $e_1$ to $e_2$ if there exists a path in the routing scheme that has $e_2$ immediately after $e_1$. We then define the dependency directed graph $\mathcal{D}=(E, D)$, whose \emph{vertices} are the channels of the original network and its directed edges represent the dependencies between channels. \cite{seminal} proved that a routing is deadlock-free if and only if its dependency graph is acyclic.

An example network, deadlock-free routing and its dependency graph are given in \Cref{fig:example-deps}.

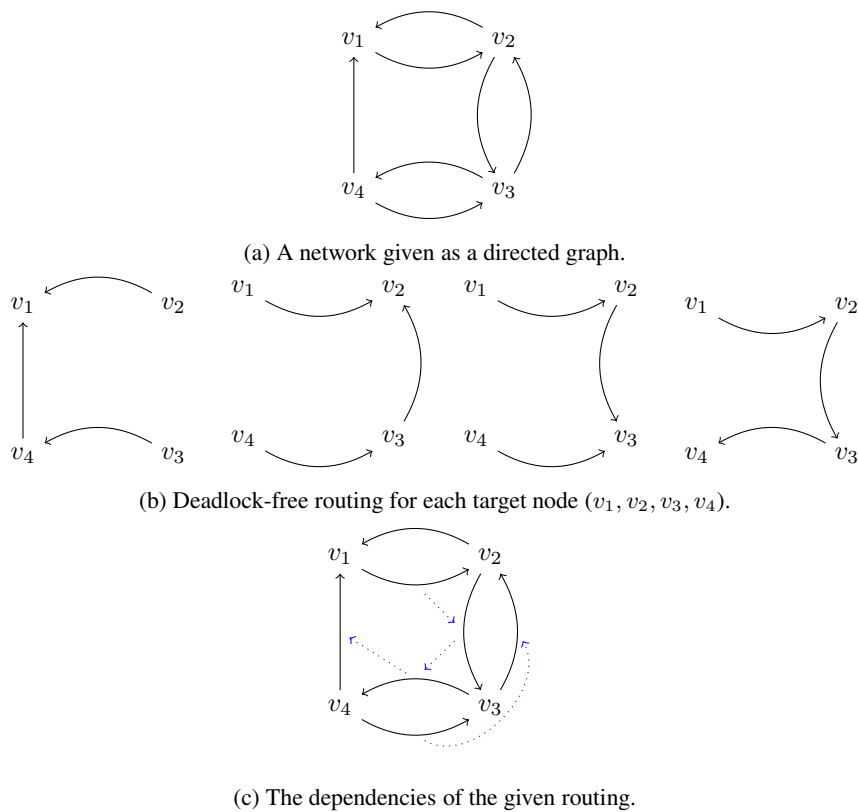
\begin{figure}[ht]
    \centering
    \begin{subfigure}{0.5\textwidth}
        \centering
        \begin{tikzpicture}[node distance=2cm]
            \node (v1) {$v_1$};
            \node[right of=v1] (v2) {$v_2$};
            \node[below of=v2] (v3) {$v_3$};
            \node[below of=v1] (v4) {$v_4$};
            
            \draw[->, bend right] (v1) to (v2);
            \draw[->, bend right] (v2) to (v3);
            \draw[->, bend right] (v3) to (v4);
            \draw[->] (v4) to (v1);
            \draw[->, bend right] (v2) to (v1);
            \draw[->, bend right] (v3) to (v2);
            \draw[->, bend right] (v4) to (v3);
        \end{tikzpicture}
        \caption{A network given as a directed graph.}
        \label{fig:network}
    \end{subfigure}
    \vfill
    \begin{subfigure}{0.95\textwidth}
        \centering
        \begin{tikzpicture}[node distance=2cm]
            \node (v1) {$v_1$};
            \node[right of=v1] (v2) {$v_2$};
            \node[below of=v2] (v3) {$v_3$};
            \node[below of=v1] (v4) {$v_4$};
            
            \draw[->, bend right] (v3) to (v4);
            \draw[->] (v4) to (v1);
            \draw[->, bend right] (v2) to (v1);
        \end{tikzpicture}
        \hfill
        \hfill
        \begin{tikzpicture}[node distance=2cm]
            \node (v1) {$v_1$};
            \node[right of=v1] (v2) {$v_2$};
            \node[below of=v2] (v3) {$v_3$};
            \node[below of=v1] (v4) {$v_4$};
            
            \draw[->, bend right] (v1) to (v2);
            \draw[->, bend right] (v3) to (v2);
            \draw[->, bend right] (v4) to (v3);
        \end{tikzpicture}
        \hfill
        \hfill
        \begin{tikzpicture}[node distance=2cm]
            \node (v1) {$v_1$};
            \node[right of=v1] (v2) {$v_2$};
            \node[below of=v2] (v3) {$v_3$};
            \node[below of=v1] (v4) {$v_4$};
            
            \draw[->, bend right] (v1) to (v2);
            \draw[->, bend right] (v2) to (v3);
            \draw[->, bend right] (v4) to (v3);
        \end{tikzpicture}
        \hfill
        \hfill
        \begin{tikzpicture}[node distance=2cm]
            \node (v1) {$v_1$};
            \node[right of=v1] (v2) {$v_2$};
            \node[below of=v2] (v3) {$v_3$};
            \node[below of=v1] (v4) {$v_4$};
            
            \draw[->, bend right] (v1) to (v2);
            \draw[->, bend right] (v2) to (v3);
            \draw[->, bend right] (v3) to (v4);
        \end{tikzpicture}
        \caption{Deadlock-free routing for each target node ($v_1, v_2, v_3, v_4$).}
        \label{fig:routing}
    \end{subfigure}
    \vfill
    \begin{subfigure}{0.5\textwidth}
        \centering
        \begin{tikzpicture}[node distance=2cm]
            \node (v1) {$v_1$};
            \node[right of=v1] (v2) {$v_2$};
            \node[below of=v2] (v3) {$v_3$};
            \node[below of=v1] (v4) {$v_4$};
            
            \draw[->, bend right] (v1) to node[midway] (e12) {} (v2);
            \draw[->, bend right] (v2) to node[midway] (e23) {}(v3);
            \draw[->, bend right] (v3) to node[midway] (e34) {}(v4);
            \draw[->] (v4) to node[midway] (e41) {}(v1);
            \draw[->, bend right] (v2) to node[midway] (e21) {}(v1);
            \draw[->, bend right] (v3) to node[midway] (e32) {}(v2);
            \draw[->, bend right] (v4) to node[midway] (e43) {}(v3);
        
            \draw[->, dotted, blue] (e34) to (e41);
            \draw[->, dotted, blue] (e43) to[out=-30,in=-60] (e32);
            \draw[->, dotted, blue] (e12) to (e23);
            \draw[->, dotted, blue] (e23) to (e34);
        \end{tikzpicture}
        \caption{The dependencies of the given routing.}
        \label{fig:dependencies}
    \end{subfigure}
    \caption{An example network, deadlock-free routing and dependency graph.}
    \label{fig:example-deps}
\end{figure}

Given a network $G=(V, E)$, one can ask whether deadlock-free routing can be constructed. We note that deadlock-free routing is always possible by adding virtual channels, splitting some of the physical channels into independent virtual ones. This was studied by \cite{routage} (this paper is written in French, see \cite{survey} for an English survey), which proved that deadlock-free routing exists for any network by splitting some of its channels to \emph{two} virtual channels. This routing, however, may have suboptimal path lengths. If the longest path length (also known as the diameter) of the graph is $d$, an optimal (in terms of path lengths) deadlock-free routing is possible by splitting each edge into at most $d$ virtual channels.

\section{Main Result}
We now ask whether deadlock-free routing can be constructed for a given network $G=(V, E)$ \emph{without adding virtual channels}. It should be mentioned that $G=(V, E)$ may contain virtual channels in the form of multiple edges between the same pair of vertices, which we consider distinct for the purpose of this analysis. We note that the desired routing is not required to achieve optimal path lengths or any other efficiency requirement.

We now state our main result.
\begin{theorem}
\label{thm:main}
Deadlock-free routing can be constructed for a network $G=(V, E)$ if and only if $G$ contains two edge-disjoint directed trees rooted at the same vertex $v$, one tree directed into $v$ and the other directed away from $v$.
\end{theorem}

\textbf{Sufficiency} of the two-tree condition of \Cref{thm:main} for the existence of deadlock-free routing follows from the fact that given such trees, any message can be routed by first sending it to the common root $v$ along the first tree, and then sending it to its target along the second tree.

The two-tree condition was known to be sufficient for deadlock-free routing. For example, \cite{routage, survey} established this fact, while \cite[Section~4.4]{method} uses such two trees to construct a deadlock-free routing for a 4x4 mesh. However, to the best of our knowledge, the necessity of this condition has neither been recognized nor proven.

To prove \textbf{necessity} we explicitly construct the two required trees from a network with a given deadlock-free routing. Using the result of \cite{seminal}, the dependency graph induced by the deadlock-free routing is acyclic, so there is a dependency-respecting total ordering of the graph edges $E$ such that every vertex is reachable from every other vertex through a sequence of edges that are ascending according to the total order.

Given a total ordering of the edges, we introduce three definitions: a global attractor, the attraction number and the attraction subgraph.

\begin{definition}
A vertex in a directed graph is called a \textbf{global attractor} if it is reachable from all other vertices.
\end{definition}

\begin{definition}
Given a strongly connected directed graph $G=(V, E)$ with a total ordering of its edges $E = (e_1, ..., e_m)$, we define its \textbf{attraction number} $s$ to be the smallest integer such that there exists a vertex $v \in V$ reachable from all other vertices using only the edges in the prefix $\{e_1, ..., e_s\}$.
\end{definition}

The existence of such an $s$ follows from the fact that $G$ is strongly connected. We note that the attraction number is the minimal size of the prefix of the edges needed for the graph to have a global attractor. For the purpose of this definition, reachability using edges $\{e_1, ..., e_s\}$ is not restricted to paths that respect the total ordering of the edges.

\begin{definition}
Given a strongly connected directed graph $G=(V, E)$ with a total ordering of its edges $E = (e_1, ..., e_m)$, we define its \textbf{attraction subgraph} to be the subgraph $G_s = (V, E_s)$ with the same vertices but only the first $s$ edges in $E$: $E_s = (e_1, ..., e_s)$, where $s$ is the attraction number of the graph.
\end{definition}

By definition, the attraction subgraph has a global attractor.

\begin{lemma}
\label{lemma:single}
Let $G=(V, E)$ be a strongly connected directed graph with a total ordering of its edges $E = (e_1, ..., e_m)$ that is consistent with the dependencies induced by a deadlock-free routing.

If there is only a single global attractor $v$ in the attraction subgraph $G_s$, then the two-trees required by \Cref{thm:main} exist in $G$.
\end{lemma}

\begin{proof}
If $v$ is the only vertex reachable from all vertices in $G_s$, no other vertex in $G_s$ can be reachable from it, so there are no outbound edges of $v$ in $G_s$. Therefore, routing messages from $v$ to other vertices according to the given deadlock-free routing cannot use any of the edges included in $G_s$, since the first step in such a route uses an edge not in $G_s$, and subsequent steps must use edges ascending according to the total order. Therefore, messages from $v$ to other vertices cannot use any edge in $G_s$, implying that routing from $v$ to all other vertices can be achieved without using edges of $G_s$. Since $v$ is reachable from all vertices using edges in $G_s$ only, and all vertices are reachable from $v$ using edges not in $G_s$, it is possible to construct two edge-disjoint trees as required by \Cref{thm:main}.
\end{proof}

\begin{lemma}
\label{lemma:multiple}
Given a strongly connected directed graph $G=(V, E)$ with a total ordering of its edges $E = (e_1, ..., e_m)$ that respects the dependencies induced by a deadlock-free routing, if there is more than one global attractor in the attraction subgraph $G_s$, then the total ordering of the edges can be modified such that the new total order still respects the deadlock-free routing, but the attraction number of the new total order is smaller than the attraction number of the original total order.
\end{lemma}

\begin{proof}
Let $R$ be the set of global attractor vertices of $G_s$. By the lemma's assumption, $R$ contains at least two vertices.

Observe that no edges in $G_s$ originate from vertices in $R$ and terminate at vertices outside of $R$, otherwise the destination vertices of such edges would have been reachable from all vertices as well, so they should have been included in $R$. We partition the edges of $G_s$ into two sets: edges between vertices in $R$ and the remaining edges of $G_s$. We denote these two sets by $C_{R \to R}$ and $C_{R \not\to R}$ respectively. Since there are no edges from $R$ to other vertices in $G_s$, the given deadlock-free routing cannot route messages along a path that uses an edge in $C_{R \to R}$ followed by an edge in $C_{R \not\to R}$.

We can use this fact to modify the original total ordering of the first $s$ edges of the original graph $G$ by placing all edges of $C_{R \not\to R}$ before the edges of $C_{R \to R}$. The order of edges within $C_{R \not\to R}$ and $C_{R \to R}$ remains unchanged. Moreover, all edges beyond the first $s$ edges retain their original place in the total order. The resulting total order still respects all dependencies of the given deadlock-free routing.

Observe that the attraction number induced by the new total order is not larger than $s$, the attraction number of the original total order, because the first $s$ edges of the new total order are exactly the same $s$ edges of the original total order, up to permutation.

\Cref{lemma:multiple} assumes there is more than one global attractor in $G_s$, so $C_{R \to R}$ cannot be empty. This means that the $s$-th edge in the new total order must connect vertices that are reachable from all vertices in $G_s$. Let $n_1$ be the source vertex of this edge. Since $n_1$ is reachable from all vertices in $G_s$ and its outbound edges cannot be used to reach it, we know that $n_1$ is also reachable from all vertices using only the first $s-1$ edges of the new total order.

We have proved that the attraction number of the new total order is at most $s-1$, proving \Cref{lemma:multiple}.
\end{proof}

We can now prove the main result using the above lemmas.

\begin{proof}[Proof of \Cref{thm:main}]
Given a deadlock-free routing for a directed graph $G=(V, E)$ we choose a specific total ordering of the edges $E$ respecting the routing dependencies that minimizes the attraction number induced by the total order.

If the attraction subgraph had more than one global attractor, we would have been able to use \Cref{lemma:multiple} to construct a new total order with a smaller attraction number, in contradiction to the fact that the chosen total order minimizes the attraction number.

Since the attraction subgraph only has a single global attractor, \Cref{lemma:single} can be applied to obtain the required trees.
\end{proof}

The flow of this proof is illustrated in \Cref{fig:proof_flow}. \Cref{fig:example-proof} illustrates the lemmas and \Cref{thm:main}'s proof through a specific example network.

The proof of \Cref{thm:main} is constructive. In fact, we can turn it into an algorithm that efficiently constructs the desired trees from a given deadlock-free routing, by repeatedly applying \Cref{lemma:multiple} until \Cref{lemma:single} can be applied.

Deciding whether a directed graph contains a pair of edge-disjoint directed trees, one towards the root and another away from the root, is known to be NP-complete \cite{np-complete}. Thus, \Cref{thm:main} shows that no efficient algorithm exists for deciding whether a deadlock-free routing can be constructed for an arbitrary network, let alone constructing such a routing.

\medskip

\printbibliography

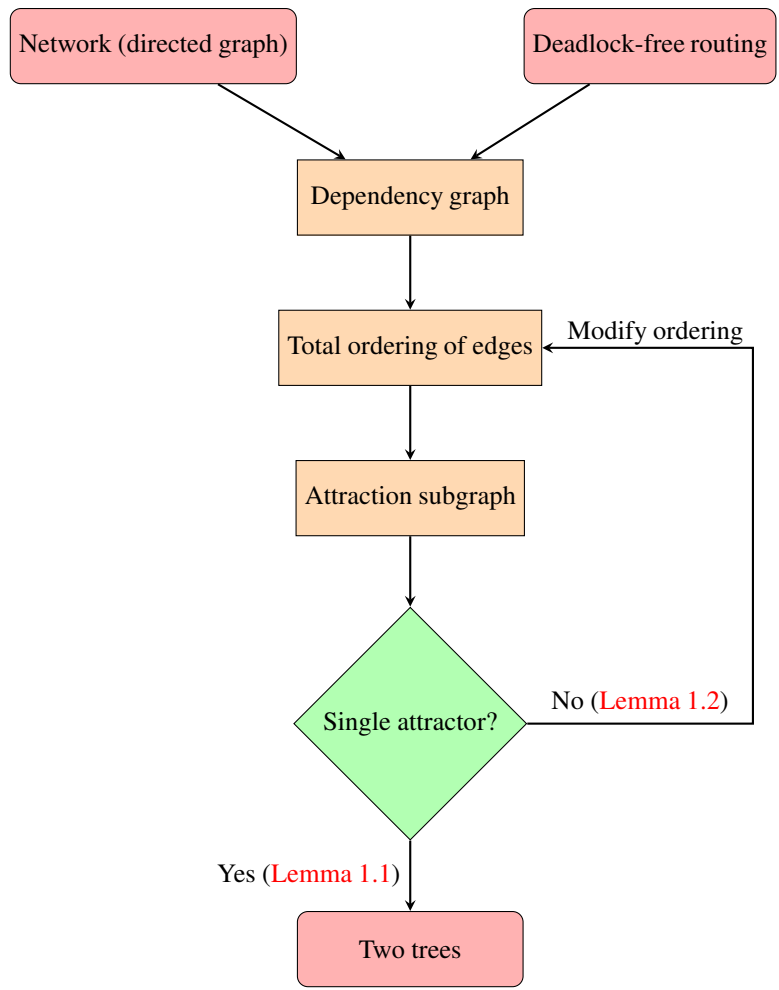
\begin{figure}[ht]
\centering
\begin{tikzpicture}[node distance=2cm]
\node (pro1) [process] {Dependency graph};
\node (start) [startstop, above left = 1cm and 0cm of pro1] {Network (directed graph)};
\node (routing) [startstop, above right = 1cm and 0cm of pro1] {Deadlock-free routing};
\node (pro2) [process, below of=pro1,yshift=0cm] {Total ordering of edges};
\node (pro3) [process, below of=pro2,yshift=0cm] {Attraction subgraph};
\node (dec1) [decision, below of=pro3,yshift=-1cm] {Single attractor?};
\node (stop) [startstop, below of=dec1,yshift=-1cm] {Two trees};

\draw [arrow] (start) -- (pro1);
\draw [arrow] (routing) -- (pro1);
\draw [arrow] (pro1) -- (pro2);
\draw [arrow] (pro2) -- (pro3);
\draw [arrow] (pro3) -- (dec1);
\draw [arrow] (dec1) -- node[anchor=east] {Yes (\Cref{lemma:single})} (stop);
\draw [arrow] (dec1.east) -- node[anchor=south] {No (\Cref{lemma:multiple})} ++(3,0) |- node[anchor=east, yshift=0.2cm] {Modify ordering} (pro2.east);
\end{tikzpicture}
\caption{Flow of the proof of \Cref{thm:main}. Applying \Cref{lemma:multiple} reduces the attraction number of the total order, guaranteeing that eventually \Cref{lemma:single} can be applied to obtain the two trees required by \Cref{thm:main}.}
\label{fig:proof_flow}
\end{figure}

\begin{figure}[ht]
    \centering
    \begin{subfigure}{0.45\textwidth}
        \centering
        \begin{tikzpicture}[node distance=2cm]
            \node (v1) {$v_1$};
            \node[right of=v1] (v2) {$v_2$};
            \node[below of=v2] (v3) {$v_3$};
            \node[below of=v1] (v4) {$v_4$};
            
            \draw[->, bend right] (v1) to node[midway,fill=white] {\textcolor{blue}{2}} (v2);
            \draw[->, bend right] (v2) to node[midway,fill=white] {\textcolor{blue}{5}} (v3);
            \draw[->, bend right] (v3) to node[midway,fill=white] {\textcolor{blue}{6}} (v4);
            \draw[->] (v4) to node[midway,fill=white] {\textcolor{blue}{7}} (v1);
            \draw[->, bend right] (v2) to node[midway,fill=white] {\textcolor{blue}{1}} (v1);
            \draw[->, bend right] (v3) to node[midway,fill=white] {\textcolor{blue}{4}} (v2);
            \draw[->, bend right] (v4) to node[midway,fill=white] {\textcolor{blue}{3}} (v3);
        \end{tikzpicture}
        \caption{A total ordering of the edges, \\ respecting the routing dependencies.}
        \label{fig:order}
    \end{subfigure}
    \hfill
    \begin{subfigure}{0.5\textwidth}
        \centering
        \begin{tikzpicture}[node distance=2cm]
            \node (v1) {\textcolor{blue}{$v_1$}};
            \node[right of=v1] (v2) {\textcolor{blue}{$v_2$}};
            \node[below of=v2] (v3) {$v_3$};
            \node[below of=v1] (v4) {$v_4$};
            
            \draw[->, bend right, blue] (v1) to node[midway,fill=white] {\textcolor{blue}{2}} (v2);
            \draw[->, bend right, blue] (v2) to node[midway,fill=white] {\textcolor{blue}{1}} (v1);
            \draw[->, bend right] (v3) to node[midway,fill=white] {4} (v2);
            \draw[->, bend right] (v4) to node[midway,fill=white] {3} (v3);
        \end{tikzpicture}
        \caption{Attraction subgraph induced by (a). \\ 
        The global attractors are colored blue.}
        \label{fig:attraction-subgraph}
    \end{subfigure}
    \vfill
    \begin{subfigure}{0.5\textwidth}
        \centering
        \begin{tikzpicture}[node distance=2cm]
            \node (v1) {$v_1$};
            \node[right of=v1] (v2) {$v_2$};
            \node[below of=v2] (v3) {$v_3$};
            \node[below of=v1] (v4) {$v_4$};
            
            \draw[->, bend right, blue] (v1) to node[midway,fill=white] {\textcolor{blue}{4}} (v2);
            \draw[->, bend right, blue] (v2) to node[midway,fill=white] {\textcolor{blue}{3}} (v1);
            \draw[->, bend right] (v3) to node[midway,fill=white] {\textcolor{blue}{2}} (v2);
            \draw[->, bend right] (v4) to node[midway,fill=white] {\textcolor{blue}{1}} (v3);
        \end{tikzpicture}
        \caption{Reordering of the subgraph edges, \\
        pushing the blue edges to the end.}
        \label{fig:attraction-subgraph-reordered}
    \end{subfigure}
    \hfill
    \begin{subfigure}{0.45\textwidth}
        \centering
        \begin{tikzpicture}[node distance=2cm]
            \node (v1) {$v_1$};
            \node[right of=v1] (v2) {$v_2$};
            \node[below of=v2] (v3) {$v_3$};
            \node[below of=v1] (v4) {$v_4$};
            
            \draw[->, bend right] (v1) to node[midway,fill=white] {\textcolor{blue}{4}} (v2);
            \draw[->, bend right] (v2) to node[midway,fill=white] {\textcolor{blue}{5}} (v3);
            \draw[->, bend right] (v3) to node[midway,fill=white] {\textcolor{blue}{6}} (v4);
            \draw[->] (v4) to node[midway,fill=white] {\textcolor{blue}{7}} (v1);
            \draw[->, bend right] (v2) to node[midway,fill=white] {\textcolor{blue}{3}} (v1);
            \draw[->, bend right] (v3) to node[midway,fill=white] {\textcolor{blue}{2}} (v2);
            \draw[->, bend right] (v4) to node[midway,fill=white] {\textcolor{blue}{1}} (v3);
        \end{tikzpicture}
        \caption{Modified total order.}
        \label{fig:order-modified}
    \end{subfigure}
    \vfill
    \begin{subfigure}{0.45\textwidth}
        \centering
        \begin{tikzpicture}[node distance=2cm]
            \node (v1) {\textcolor{blue}{$v_1$}};
            \node[right of=v1] (v2) {$v_2$};
            \node[below of=v2] (v3) {$v_3$};
            \node[below of=v1] (v4) {$v_4$};
            
            \draw[->, bend right] (v2) to node[midway,fill=white] {\textcolor{blue}{3}} (v1);
            \draw[->, bend right] (v3) to node[midway,fill=white] {\textcolor{blue}{2}} (v2);
            \draw[->, bend right] (v4) to node[midway,fill=white] {\textcolor{blue}{1}} (v3);
        \end{tikzpicture}
        \caption{Modified attraction subgraph with a single global attractor ($v_1$). These edges can be used to construct the inbound tree.}
        \label{fig:attraction-subgraph-modified}
    \end{subfigure}
    \hfill
    \begin{subfigure}{0.45\textwidth}
        \centering
        \begin{tikzpicture}[node distance=2cm]
            \node (v1) {\textcolor{blue}{$v_1$}};
            \node[right of=v1] (v2) {$v_2$};
            \node[below of=v2] (v3) {$v_3$};
            \node[below of=v1] (v4) {$v_4$};
            
            \draw[->, bend right] (v1) to node[midway,fill=white] {\textcolor{blue}{4}} (v2);
            \draw[->, bend right] (v2) to node[midway,fill=white] {\textcolor{blue}{5}} (v3);
            \draw[->, bend right] (v3) to node[midway,fill=white] {\textcolor{blue}{6}} (v4);
        \end{tikzpicture}
        \caption{The outbound tree from $v_1$ to all vertices,
        using edges not used for the modified attraction subgraph.}
        \label{fig:inbound-tree}
    \end{subfigure}
    \caption{An application of \Cref{thm:main} to the example network and routing given in \Cref{fig:example-deps}.}
    \label{fig:example-proof}
\end{figure}
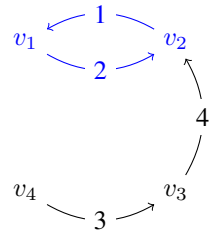
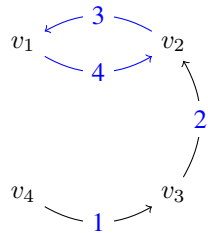
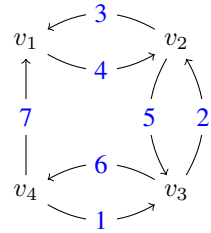
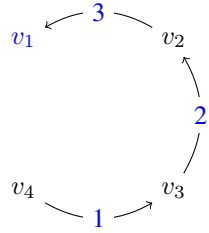
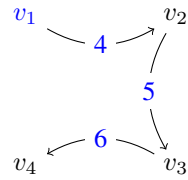

\end{document}